\documentclass{article}
\usepackage[mathscr]{eucal}
\usepackage{amsmath}
\usepackage{amsthm}
\usepackage{amsfonts}
\usepackage{color}
\usepackage{url}

\AtEndDocument{\bigskip{\footnotesize%
\par
  (M.~Kulczycki) \textsc{Institute of Mathematics, Faculty of Mathematics and Computer Science, Jagiellonian University in Krak\'{o}w, ul. {\L}ojasiewicza 6, 30-348 Krak\'{o}w, Poland} \par
  \textit{E-mail address}, M.~Kulczycki: \texttt{Marcin.Kulczycki@im.uj.edu.pl} \par
  \addvspace{\medskipamount}
  (D.~Kwietniak) \textsc{Institute of Mathematics, Faculty of Mathematics and Computer Science, Jagiellonian University in Krak\'{o}w, ul. {\L}ojasiewicza 6, 30-348 Krak\'{o}w, Poland} \par
  \textit{E-mail address}, D.~Kwietniak: \texttt{dominik.kwietniak@uj.edu.pl} \par
  \addvspace{\medskipamount}
  (J.~Li) \textsc{Department of Mathematics, Shantou University, Shantou, Guangdong 515063, P.R. China} \par
  \textit{E-mail address}, J.~Li: \texttt{lijian09@mail.ustc.edu.cn
}
}}

\newtheorem{theorem}{Theorem}
\newtheorem{lemma}{Lemma}

\newtheorem*{fekete}{Fekete's Lemma}

\theoremstyle{definition}
\newtheorem*{definition}{Definition}
\newtheorem*{definitions}{Definitions}

\newtheorem*{exercise}{Exercise}
\newtheorem*{example}{Example}

\newcommand{\N}{\mathbb{N}}
\newcommand{\Al}{\mathscr{A}}
\newcommand{\Bl}{\mathscr{B}}
\newcommand{\empw}{\varepsilon}

\newcommand{\F}{\mathscr{F}}

\begin{document}

\title{Entropy of subordinate shift spaces}
\markright{Entropy of subordinate subshifts}
\author{Marcin Kulczycki, Dominik Kwietniak, and Jian Li}

\maketitle

\begin{abstract}
We introduce a new family of shift spaces --- the subordinate shifts. Using subordinate shifts we prove in an elementary way that for every nonnegative real number $t$ there is a shift space with entropy $t$.
\end{abstract}

\section{Introduction.}

Let $\Al$ be a finite collection of \emph{symbols} which we call an \emph{alphabet}. Typical choices are $\Al=\{0,1,\ldots,r-1\}$ for some integer $r$ or
a set of Roman letters, for example $\Al=\{a,b,c,d\}$. The \emph{full shift} over $\Al$ is the set $\Al^\N$ of all infinite sequences of symbols. A \emph{block} is any finite sequence of symbols.

A \emph{shift space} over $\Al$ is a subset of $\Al^\N$ defined by establishing some constraints on blocks, which are allowed to appear as subsequences.
For example, the set of all infinite sequences of $0$'s and $1$'s that do not contain the block $11$ is a shift space over $\{0,1\}$. A shift space becomes a dynamical system when we equip it with the \emph{shift map}, which discards the first element of a sequence and shifts the remaining elements by one place to the left. The branch of mathematics concerning the study of shift spaces is known as \emph{symbolic dynamics}.

Shift spaces are mathematical models for digitized information, arising frequently as a result of the discretization of dynamical processes. For example, imagine a point moving along a trajectory in space. Partition the space into finitely many pieces and assign a symbol to each piece. Now write down the sequence of symbols labelling the successive partition elements visited by the point as it follows the trajectory. We have encoded information about the trajectory in a sequence from the shift space over the set of labels of the partition cells.

As is often the case in mathematics, there is a notion of isomorphism for shift spaces known as \emph{conjugacy} showing
us when two seemingly different shift spaces are practically the same. In other words, conjugate shift spaces can be treated as two instances of the same underlying object (for more details we refer the reader to \cite[Def. 1.5.9]{LM95}).
For example, there is neither loss nor gain of information when the full $\{0,1\}$-shift is transformed into the full $\{4,7\}$-shift by switching every $0$ to $4$ and every $1$ to $7$.

A common problem in symbolic dynamics is  to decide whether two shift spaces are conjugate or not. Various invariants are devised to help us distinguish shift spaces which are \emph{not} conjugate. A property of a shift space is a \emph{conjugacy invariant} if whenever a shift space $X$ possesses that property, then every shift space conjugate to $X$ also possesses that property. Among the most important invariants is \emph{entropy}, which is a measure of the complexity of a shift space. Entropy is a nonnegative real number equal to the asymptotic growth rate of the number of blocks that occur in a shift space.

Entropy is a conjugacy invariant \cite[Cor. 4.1.10]{LM95}. Therefore, if two shift spaces have different entropy, then they have different structure  --- they are not conjugate. For example, the full shift over $\{0,1\}$ is not conjugate to the full  shift over $\{a,b,c\}$, as they have different entropies.

There is one problem with this useful tool: computing the entropy of an arbitrary shift space is often a difficult or even a hopeless task. For example, it is known that for every nonnegative real number $t$ there is a shift space with entropy $t$, but this result is a part of nontrivial theory (see for example  \cite[pp. 178--9]{Walters}).

We will show, however, that this can be proven in an elementary way. To this end, we will define a new class of shift spaces --- the subordinate shifts --- for some of which the calculation of entropy is straightforward and requires only basic combinatorics.

\subsection{Notation and definitions.}

Throughout this paper, the symbol $\N$ denotes the set of \textit{positive} integers. We denote the number of elements of a finite set $A$ by $|A|$. Given any real number $x$, by $\lfloor x\rfloor$ we mean the largest integer not greater than $x$.

Recall that a sequence of real numbers $\{a_n\}_{n=1}^\infty$ is \textit{subadditive} if
$a_{m+n}\leq a_m+a_n$ for all $m,n\in\N$.

\begin{fekete}
Let $\{a_n\}_{n=1}^\infty$ be a subadditive sequence of nonnegative real numbers. Then the sequence $\{a_n/n\}_{n=1}^\infty$ converges to a limit equal to the infimum of the terms of this sequence, that is
\[
\lim_{n\to\infty}\frac{a_n}{n}  = \inf_{n\in\N}\frac{a_n}{n}.
\]
\end{fekete}

We follow the notation of Lind and Marcus~\cite{LM95} as close as possible. One notable exception is that we consider only one-sided shifts, while Lind and Marcus consider two-sided (invertible) shifts throughout most of their book.

\begin{definitions}[Full shifts]
Let $\Al$ be a finite set, which we call the  \emph{alphabet}. We refer to the elements of $\Al$ as \emph{symbols}. The \emph{full $\Al$-shift} is the collection of all infinite sequences of symbols from $\Al$. The full $\Al$-shift is denoted by
\[\Al^\N=\{x=(x_i)_{i=1}^\infty : x_i\in\Al\text{ for all } i\in\N\}.\]
We usually write an element of $\Al^\N$ as $x=(x_i)_{i=1}^\infty=x_1x_2x_3\ldots$.
Often we identify a finite set $\Al$ such that $|\Al|=r$ with $\{0,1,\ldots,r-1\}$. A \emph{full $r$-shift} is then the full shift over the alphabet $\{0,1,\ldots,r-1\}$ and a \emph{full binary shift} is the full $2$-shift.
\end{definitions}

\begin{definitions}[Blocks]
A \emph{block over $\Al$} is a finite sequence of symbols from $\Al$. We write blocks without separating their symbols, so a block over $\Al=\{0,1,2\}$ might look like $01220120$. The \emph{length of a block $u$} is the number of symbols it contains. An \emph{$n$-block} stands for a block of length $n$. We identify a symbol with the $1$-block consisting of this symbol. An \emph{empty block} is the unique block with no symbols and length zero that we denote $\empw$. The set of all blocks over $\Al$ (including $\empw$) is denoted by $\Al^*$.

A \emph{concatenation} of two blocks $u=a_1\ldots a_k$ and $v=b_1\ldots b_l$ is the block $uv$ obtained by writing $u$
first and then $v$, that is, $uv=a_1\ldots a_k b_1\ldots b_l$. The concatenation is an associative operation, because $(uv)w=u(vw)$ for any blocks $u,v,w\in\Al^*$. For this reason we may write $uvw$, or indeed concatenate any sequence of blocks (finite or not) without ambiguity. If $n\ge 1$, then $u^n$ stands for the concatenation of $n$ copies of $u$. Given a nonempty block $u\in\Al^*$ we denote by $u^\infty$ the sequence $uuu\ldots\in\Al^\N$.

Let $x=(x_i)_{i=1}^\infty\in\Al^\N$ and let $1\le i \le j$ be integers. We write $x_{[i,j]}=x_ix_{i+1}\ldots x_j$ for the  block of symbols in $x$ starting from the $i$-th and ending at the $j$-th position. We say that a block $w\in A^*$ \emph{occurs in $x$} and $x$ \emph{contains} $w$ if $w=x_{[i,j]}$ for some integers $1\le i \le j$. Note that $\empw$ occurs in every sequence from $\Al^\N$. Similarly, given an $n$-block $w=w_1\ldots w_n \in\Al^*$ we define $w_{[i,j]}=w_iw_{i+1}\ldots w_j\in\Al^*$ for each $1\le i\le j\le n$.
\end{definitions}

A \emph{prefix} of a block $z\in\Al^*$ is a block $u$ such that $z=uv$ for some $v\in\Al^*$.

\begin{definition}[Shift map]
For every $\Al$ we define the \emph{shift map} $\sigma\colon\Al^\N\rightarrow\Al^\N$. It maps a sequence $x=(x_i)_{i=1}^\infty$ to the sequence $\sigma(x)=(x_{i+1})_{i=1}^\infty$. Equivalently, $\sigma(x)$ is the sequence obtained by dropping the first symbol of $x$ and moving the remaining symbols by one position to the left.
\end{definition}

\begin{definitions}[Shift spaces]
Given any collection $\F$ of blocks over $\Al$ (i.e., a subset of $\Al^*$) we define a \emph{shift space specified by $\F$}, denoted by $X_\F$, as the set of all sequences from $\Al^\N$ which do not contain any blocks from $\F$. We say that $\F$ is a collection of \emph{forbidden blocks for $X_\F$} as blocks from $\F$ are forbidden to occur in $X_\F$.

A \emph{shift space} is a set $X\subset\Al^\N$ such that $X=X_\F$ for some $\F\subset\Al^*$. A \emph{binary shift space} is a shift space over the alphabet $\{0,1\}$.
\end{definitions}
\begin{exercise}
Show that for every shift space $X$ we have $\sigma(X)\subset X$. Find a shift space $X$ for which $\sigma(X)\neq X$.
\end{exercise}
\begin{definition}[Language of a shift space]
With each shift space $X$ over $\Al$, we may associate a set of blocks over $\Al$ which occur in some sequence $x\in X$. We call this set the \emph{language of $X$} and denote it by $\Bl(X)$. We write $\Bl_n(X)$ for the set of all $n$-blocks contained in $\Bl(X)$.
\end{definition}

\begin{exercise}
Show that if $\mathcal{L}$ is a language of some shift space over $\Al$, then $\mathcal{L}$ is:
\begin{enumerate}
\item \emph{factorial}, meaning that if $u\in\mathcal{L}$ and $u=vw$ for some blocks $v,w\in \Al^*$, then both $v$ and $w$ also belong to $\mathcal{L}$,
\item \emph{prolongable}, meaning that for every block $u$ in $\mathcal{L}$ there is a symbol $a\in \Al$ such that $ua$ also belongs to $\mathcal{L}$.
\end{enumerate}
\end{exercise}
Actually, the converse is also true. Given a factorial and prolongable subset $\mathcal{L}\subset\Al^*$ there is a shift space $X$ such that $\mathcal{L}$ is the language of $X$. A collection of forbidden blocks defining $X$ is $\F=\Al^*\setminus\mathcal{L}$.

We can also characterize points in a shift space $X$.

\begin{lemma}\label{lem:char}
Let $\mathcal{L}\subset\Al^*$ be factorial and prolongable. Let $X$ be a shift space such that $\mathcal{L}=\Bl(X)$.
Then a point $x\in \Al^\N$ is in $X$ if and only if $x_{[i,j]}\in\mathcal{L}$ for all $i,j\in\N$ with $ i < j$.
\end{lemma}

Shift spaces are determined by their language. In other words, two shift spaces are equal if and only if they have the same language \cite[Proposition 1.3.4]{LM95}. Hence there is a one-to-one correspondence between shift spaces over $\Al$ and factorial, prolongable subsets of $\Al^*$. To define a shift space, one can either specify its set of forbidden blocks or its language.

\begin{example}
Let $x\in\Al^\N$. Let $\Bl_x$ be the collection of all blocks occurring in $x$. Since $\Bl_x$ is a factorial and prolongable language, it defines a shift space, which we denote by $\Sigma_x$.
\end{example}

\begin{example}
Let $S\subset\N\cup\{0\}$. We define
$$\F_S=\{1\underbrace{0\ldots 0}_p1:p\notin S\}.$$
The binary shift defined by forbidding blocks from $\F_S$ is called the \emph{$S$-gap shift} and is denoted by $X_S$. In particular, we call $X_\N$ the \emph{golden mean shift} (a sequence belongs to it if and only if it does not contain the block $11$).
\end{example}

The following fact merely states that we may construct inductively an \emph{infinite} sequence starting from an infinite collection of \emph{finite} sequences such that each sequence in it coincides with any shorter one as long as both are defined.

\begin{lemma}\label{lem:growing-words}
Let $\{w^{(n)}\}_{n=1}^\infty$ be a sequence in $\Al^*$ and for each $n\in\N$ let $l(n)$ be the length of $w^{(n)}$. If for each $k\in\N$ the block $w^{(k)}$ is a prefix of $w^{(k+1)}$, then there is a point $x\in\Al^\N$ such that for each $n\in\N$ we have $x_{[1,l(n)]}=w^{(n)}$. Moreover, if $\lim_{n\rightarrow\infty} l(n)=\infty$, then $x$ is unique.
\end{lemma}

We use superscripts in brackets as above to denote indices for sequences of blocks. That is, we write $\{w^{(n)}\}_{n=1}^\infty$ to denote a sequence of blocks. This way we may reserve subscripts for enumerating symbols within the block: $w^{(n)}=w^{(n)}_1w^{(n)}_2\ldots w^{(n)}_k$.

Finally, we are ready to move to the entropy itself. In full generality, this concept was defined by Adler, Konheim and McAndrew \cite{AKM} for an arbitrary compact topological space $X$ and a continuous map $f\colon X\to X$. The definition below applies only to shift spaces but the resulting number is equal to the Adler, Konheim and McAndrew entropy of the shift treated as a dynamical system (see \cite[Exercise 6.3.8]{LM95}).

\begin{definition}[Entropy]
Let $\log$ denote the logarithm to base $2$ (choosing a different base would also yield a valid definition; it would change the value of entropy only by a multiplicative constant). Let $X\subset\Al^\N$ be a nonempty shift space and let $m,n\in\N$. Observe that every block $w\in\Bl_{m+n}(X)$ can be written in a unique way as a concatenation $w=uv$, where $u\in \Bl_m(X)$ and $v\in \Bl_n(X)$. Therefore $|\Bl_{m+n}(X)|\le|\Bl_{m}(X)|\cdot |\Bl_{n}(X)|$, and hence
\[
\log |\Bl_{m+n}(X)|\le \log|\Bl_{m}(X)|+\log |\Bl_{n}(X)|.
\]
By applying Fekete's Lemma to the nonnegative sequence $\log|\Bl_n(X)|$ we may now define the \emph{entropy of $X$}, denoted by $h(X)$, as
\[
h(X)=\lim_{n\to\infty} \frac{1}{n}\log |\Bl_n(X)|=\inf_{n\ge 1} \frac{1}{n}\log |\Bl_n(X) |.
\]
\end{definition}

Roughly speaking, the entropy measures the complexity of a shift space $X$ in terms of the asymptotic growth rate of the number of $n$-blocks that appear in the language $X$. In other words, the number of $n$-blocks in a shift space of entropy $h\ge 0$ roughly equals $2^{nh}$.

\begin{example}
Every finite shift space has entropy zero.
\end{example}

\begin{example}
The full $\Al$-shift has entropy $\log|\Al|$. In particular, the full $r$-shift has entropy $\log r$.
\end{example}

Observe that if $X,Y\subset\Al^\N$ and $X\subset Y$, then $h(X)\leq h(Y)$. Since a shift space over $\Al$ is a subset of the full $\Al$-shift, we may conclude that the entropy of any shift space over $\Al$ is a nonnegative real number bounded above by $\log|\Al|$ (for systems that are not shifts it may well be infinite --- see \cite[Example 4.2.6]{ALM}).

As mentioned in the introduction, computing the entropy of a shift space is a hard problem --- try, for example to verify straight from the definition that the entropy of the golden mean shift is $\log((1+\sqrt{5})/2)$. There are (relatively rare) families of shift spaces (e.g. shifts of finite type, see \cite{LM95}) for which we can actually provide a (theoretically) computable formula for entropy. Even in these special cases, one needs to apply some non-trivial tools. For the sake of illustration we recall some results from \cite[Exercise 4.3.7]{LM95}.

\begin{theorem}
Let $S\subset\N\cup\{0\}$ and let $X_S$ be the associated $S$-gap shift. Then $h(X_S)=\log\lambda$, where $\lambda$ is the unique positive solution of the equation
\[
\sum_{j\in S} x^{-j-1}=1.
\]
\end{theorem}

\begin{theorem}
For every $t\in[0,1]$ there is a set $S\subset\N\cup\{0\}$ ($S$ depends on $t$) such that $h(X_S)=t$.
\end{theorem}

\section{Subordinate shifts.}

The goal of this section is to introduce a family of shift spaces with easily calculable entropy; a family rich enough to contain a shift space with every possible nonnegative entropy. For the remainder of the paper, we fix $\Al=\{0,1,\ldots, r-1 \}$.

\begin{definition}[Subordinate shift]
We say that a block $w=w_1\ldots w_k\in\Al^*$ \emph{dominates} a block $v=v_1\ldots v_k\in\Al^*$ if $v_i\le w_i$ for $i=1,\ldots,k$. In an analogous way we define when one sequence from $\Al^\N$ dominates another.

A \emph{subordinate of $\mathcal{L}\subset \Al^*$} is the set $\mathcal{L}^\le$ of all blocks over $\Al$ that are dominated by some block in $\mathcal{L}$. Observe that if $\mathcal{L}$ is factorial and prolongable, then the same holds for $\mathcal{L}^\le$. In particular, given a point $x\in\Al^\N$, we may define a \emph{subordinate shift of $x$}, denoted by $X^{\le x}$, as a shift space given by the language $\Bl^\le_x$, where $\Bl_x$ is the language of blocks occurring in $x$.
\end{definition}

\begin{example}
All binary blocks of length $3$ are dominated by $111$. The blocks $0000$, $0001$, $0100$, and $0101$ are the only blocks dominated by $0101$.
\end{example}

Subordinate shifts are \emph{hereditary} (this is a notion introduced in \cite{KerrLi} and examined in \cite{Kwietniak}).
It can be shown that a hereditary shift is subordinate if and only if it is irreducible in the sense of \cite[Definition 1.3.6]{LM95}.
It turns out that a shift space that has been recently extensively studied is an example of a subordinate shift.
\begin{example}
Recall that a positive integer $n$ is \emph{square-free} if there is no prime number $p$ such that $p^2$ divides $n$. Let $\eta$ be a point in $\{0,1\}^\N$
given by
\[
\eta_n=\begin{cases}
1&\text{ if $n$ is square-free,}\\
0& \text{ otherwise.}
\end{cases}
\]
In other words, $\eta_n=(\mu(n))^2$, where $\mu\colon\N\to\N$ is the famous M\"{o}bius function.
It can be shown that
$S=X^{\le \eta}$ is the \emph{square-free flow}; that is a shift space, whose structure is strongly tied to the statistical properties
of square-free numbers. For more details see \cite{Peckner,Sarnak}. The study of the square-free flow has been recently extended to the more general
context of $\mathcal{B}$-free integers; that is to say integers with no factor in a given
family $\mathcal{B}$ of pairwise relatively prime integers, the sum of whose reciprocals is
finite, see \cite{B-free-dynamical,B-free-measures}.
\end{example}

We aim to show that if given $t\in [0,1]$, then we are able to choose a point $x(t)$ from the full binary shift such that
$h(X^{\le x(t)})=t$. First we tackle rational entropies.

\begin{lemma}\label{lem:rational-case}
If $w\in\{0,1\}^*$ is a block of length $q$ with $p$ occurrences of the symbol $1$ and $x=w^\infty$, then $h(X^{\le x})=p/q$.
\end{lemma}

\begin{proof}
Since replacing in $w$ any subset of $1$'s with $0$'s leads to a block dominated by $w$ which is in $\Bl(X^{\le x})$, we know that there are at least $2^p$ blocks in $\Bl_q(X^{\le x})$. It follows that $h(X^{\le x})\geq p/q$.

On the other hand, the periodicity of $x=w^\infty$ implies that for each $j\in\N$ there are at most $q$ different blocks of length $qj$ in $\Bl_{x}$. Each such block dominates exactly $2^{pj}$ blocks from $\{0,1\}^*$. Therefore there are at most $q\cdot 2^{pj}$ blocks in $\Bl_{qj}(X^{\le x})$. Consequently,
\begin{align*}
h(X^{\le x}) & =  \lim_{n\to\infty} \frac{1}{n}\log|\Bl_n(X^{\le x})| = \lim_{j\to\infty} \frac{1}{qj}\log|\Bl_{qj}(X^{\le x})|\\
& \le  \lim_{j\to\infty} \frac{1}{qj}\log (q\cdot 2^{pj})=p/q.\qedhere
\end{align*}
\end{proof}

We now show how to construct a shift space with entropy $\pi/8$. We hope that this example will make the general construction that comes after it much clearer.

\begin{example}
Let $t=\pi/8=0.3926990816\ldots$. We consider the following sequence of rational approximations of $t$ from the above:
\[
0.4,\, 0.4,\, 0.393,\, 0.3927,\,0.3927,\,0.3927,\,0.3926991,\,0.39269909,\,\ldots .
\]
We obtain our $n$-th approximation by rounding up $t$ to the nearest number with no more than $n$ digits after the decimal point. Let $p_n$ be this $n$-th approximation times $10^n$.

We now inductively build a sequence of blocks $\{w^{(n)}\}_{n=1}^\infty$ from $\{0,1\}^*$ such that for every $n\in\N$:
\begin{enumerate}
\item $w^{(n)}$ has length $10^n$,
\item the symbol $1$ appears exactly $p_n$ times in $w^{(n)}$,
\item $w^{(n)}$ is a prefix of $w^{(n+1)}$,
\item $w^{(n+1)}$ is dominated by $(w^{(n)})^{10}$.
\end{enumerate}

We can visualize the construction of $w^{(n+1)}$ as a process with two steps. In the first step we concatenate ten copies of $w^{(n)}$. In the second step we keep the first $p_{n+1}$ occurrences of the symbol $1$ in $w^{(n+1)}$ and replace the rest by $0$'s.

In our example, we first put $w^{(1)}=1111000000$. We want the symbol $1$ to appear $p_2=40$ times in $w^{(2)}$, so we define $w^{(2)}=(w^{(1)})^{10}$ --- there is no need to remove any $1$'s. Then in $w^{(3)}$ we need to see $393$ appearances of $1$, so we define $w^{(3)}$ as a concatenation of nine copies of $w^{(2)}$ and a block $v$ of length $100$ which agrees with $w^{(2)}$ except that the last seven $1$'s appearing in $w^{(2)}$ are replaced by $0$'s in $v$. The construction continues on inductively.

Applying Lemma \ref{lem:growing-words} to the sequence of binary blocks $\{w^{(n)}\}_{n=1}^\infty$,
we obtain a point $x\in\{0,1\}^\N$. The entropy of the subordinate shift $X^{\le x}$ is $t=\pi/8$. The proof of this fact is contained in the general result below.
\end{example}

We are now equipped to tackle the main theorem of this paper.

\begin{theorem}\label{thm:binary-case}
For every $t\in[0,1]$ there is a binary subordinate shift with entropy $t$.
\end{theorem}

\begin{proof}
For every $n\in\N$ let $s_n$ be the rational approximation of $t$ obtained by rounding up the decimal expansion of $t$ to the nearest number with no more than $n$ digits after the decimal point. Let $p_n=s_n\cdot 10^n$.

We now inductively build a sequence of blocks $\{w^{(n)}\}_{n=1}^\infty$ from $\{0,1\}^*$ such that for every $n\in\N$:
\begin{enumerate}
\item $w^{(n)}$ has length $10^n$,
\item the symbol $1$ appears exactly $p_n$ times in $w^{(n)}$,
\item $w^{(n)}$ is a prefix of $w^{(n+1)}$,
\item $w^{(n+1)}$ is dominated by $(w^{(n)})^{10}$.
\end{enumerate}

We start with
\[
w^{(1)}=\underbrace{1\ldots1}_{p_1}\underbrace{0\ldots 0}_{10-p_1}.
\]

Assume now that we have defined $w^{(1)},\ldots,w^{(n)}$ so that the four conditions above are satisfied to the extent to which they apply to $w^{(1)},\ldots,w^{(n)}$. We now concatenate ten copies of $w^{(n)}$. We keep the first $p_{n+1}$ occurrences of the symbol $1$ unaltered in the sequence, replace the rest by $0$, and call the result $w^{(n+1)}$. It is easy to verify that the conditions above now hold to the extent to which they apply to $w^{(1)},\ldots,w^{(n+1)}$, so the inductive construction is complete.

Applying Lemma \ref{lem:growing-words} to the sequence $\{w^{(n)}\}_{n=1}^\infty$ we obtain a point $x\in\{0,1\}^\N$.

Observe that for every $n\in\N$ the point $x$ is dominated by $(w^{(n)})^\infty$. It follows that $X^{\leq x}\subset X^{\le (w^{(n)})^\infty}$, and so $h(X^{\leq x})\leq h(X^{\le (w^{(n)})^\infty})$. By Lemma \ref{lem:rational-case} we have $h(X^{\le (w^{(n)})^\infty})=s_n$, and therefore
\[h(X^{\leq x})\leq\lim_{n\to\infty}h(X^{\le (w^{(n)})^\infty})=\lim_{n\to\infty}s_n=t.\]

On the other hand, observe that for every $n\in\N$ we have $w^{(n)}\in\Bl_{10^n}(X^{\leq x})$, so all $2^{p_n}$ blocks dominated by $w^{(n)}$ are also in $\Bl_{10^n}(X^{\leq x})$. Therefore $\log |\Bl_{10^n}(X^{\leq x})|\ge \log 2^{p_n} = p_n$, and so
\[
t=\lim_{n\to\infty}s_n=\lim_{n\to\infty}\frac{p_n}{10^n}\le \lim_{n\to\infty} \frac{1}{10^n}\log |\Bl_{10^n}(X^{\leq x})|=h(X^{\leq x}),
\]
which completes the proof of $h(X^{\leq x})=t$.
\end{proof}

\begin{exercise}
Prove that the shift space $X^{\leq x}$ constructed in the proof of Theorem \ref{thm:binary-case} is \emph{irreducible} (see \cite[Definition 1.3.6]{LM95}), meaning that given any pair of blocks $u,v\in\Bl(X^{\leq x})$ there is a block $y\in\Bl(X^{\leq x})$ such that $uyv\in\Bl(X^{\leq x})$.
\end{exercise}

It remains to show that the conclusion of Theorem \ref{thm:binary-case} is true for every $t> 1$.
\begin{theorem}
For every $t\in (1,\infty)$ there is a shift space with entropy $t$.
\end{theorem}

\begin{proof}
Let $t\in (1,\infty)$ be given. Pick $k\in\N$ and $s\in (0,1)$ such that $ks=t$. Use Theorem \ref{thm:binary-case} to obtain a binary subordinate shift $X$ with entropy $s$. Using $X$ we now construct a shift space $Y\subset\{0,1,\ldots,2^k-1\}^\N$.

We start with $Y=\emptyset$. For every point $w\in X$ we perform the following procedure:
\begin{enumerate}
\item express $w$ as a concatenation $w_1w_2\ldots$, where each block $w_n$ has length $k$,
\item for every $w_n$ we determine the number $a_n\in\{0,1,\ldots,2^k-1\}$ such that $w_n$ is the binary notation for $a_n$ (for example, for $k=2$, we get $00\mapsto 0$, $01\mapsto 1$, $10\mapsto 2$, and $11\mapsto 3$),
\item we add the sequence $a_1a_2\ldots$ to $Y$ (note that it is the $a_n$'s that are symbols here, not their digits).
\end{enumerate}
It is elementary to check that $Y$ is a shift space. It suffices to analyze how the language of $Y$ is created from the language of $X$.

Observe that for every $n\in\N$ we have $|\Bl_n(Y)|=|\Bl_{kn}(X)|$. Therefore
\[
h(Y)=\lim_{n\to\infty}\frac{1}{n}\log |\Bl_n(Y)|=k\cdot\lim_{n\to\infty}\frac{1}{kn}\log |\Bl_{kn}(X)|=k\cdot h(X)=t.\qedhere
\]
\end{proof}

\section*{Acknowledgment.}
The authors would like to thank the referees for their thorough and careful work. Preparing this article we asked our students
and colleagues to comment on it. We are grateful to: Jakub Byszewski, Vaughn Climenhaga, Jakub Konieczny, Marcin Lara, Simon Lunn, Martha {\L}{\c{a}}cka, Dariusz Matlak, Samuel Roth, and Maciej Ulas
for their remarks and suggestions.
The research of Dominik Kwietniak was supported by the  National Science Centre (NCN) under grant Maestro 2013/08/A/ST1/00275.
The research of Jian Li  was supported by Scientific Research Fund of Shantou University (YR13001).

\end{document}